\documentclass[11pt]{article}

\usepackage{enumerate}
\usepackage{amsthm,amsmath,amssymb}
\usepackage{graphicx}
\usepackage{lineno}
\usepackage[colorlinks=true,citecolor=black,linkcolor=black,urlcolor=blue]{hyperref}
\usepackage[english]{babel}
\usepackage{amsfonts}
\usepackage{epsfig, subfigure}
\usepackage{amscd,latexsym}
\usepackage{url}
\usepackage{color}
\usepackage{authblk}
\usepackage{subfigure}

\theoremstyle{plain}
\newtheorem{theorem}{Theorem}
\newtheorem{lemma}[theorem]{Lemma}
\newtheorem{cor}[theorem]{Corollary}
\newtheorem{prop}[theorem]{Proposition}

{\itshape}{\rmfamily}
\newtheorem{remark}[theorem]{Remark}

\newtheorem{conj}[theorem]{Conjecture}

\theoremstyle{plain}

\newcommand\blfootnote[1]{%
	\begingroup
	\renewcommand\thefootnote{}\footnote{#1}%
	\addtocounter{footnote}{-1}%
	\endgroup
}

\newcommand{\kftrs}{$(k+1)$-resolving set}
\newcommand{\kftrss}{$(k+1)$-resolving sets}
\newcommand{\kfftrss}{Resolving sets tolerant to $k$ failures}
\newcommand{\kftmd}{$(k+1)$-metric dimension}

\newcommand{\threedim}{3D }

\begin{document}
	
	\title{
		Resolving sets tolerant to failures in three-dimensional grids
	}
	
	\author[1]{Merc\`e Mora\thanks{Partially supported by projects H2020-MSCA-RISE-2016-734922 CONNECT,
			PID2019-104129GB-I00/MCIN/AEI/10.13039/501100011033
			of the Spanish Ministry of Science and Innovation
			and Gen.Cat. DGR2017SGR1336, merce.mora@upc.edu}}
	
	\author[2]{Mar\'ia Jos\'e Souto-Salorio\thanks{Partially supported by project PID2020-113230RB-C21 of the Spanish Ministry of Science and Innovation, maria.souto.salorio@udc.es}}
	
	\author[2]{Ana Dorotea Tarr\'io-Tobar\thanks{ana.dorotea.tarrio.tobar@udc.es}}
	
	\affil[1]{Departament de Matem\`atiques, Universitat Polit\`ecnica de Catalunya, Spain}
	
	\affil[2]{Departamento de Matem\'aticas, Universidade da Coru\~{n}a\\ Spain}

	\date{}
	\maketitle
	
	\blfootnote{\begin{minipage}[l]{0.3\textwidth} \includegraphics[trim=10cm 6cm 10cm 5cm,clip,scale=0.15]{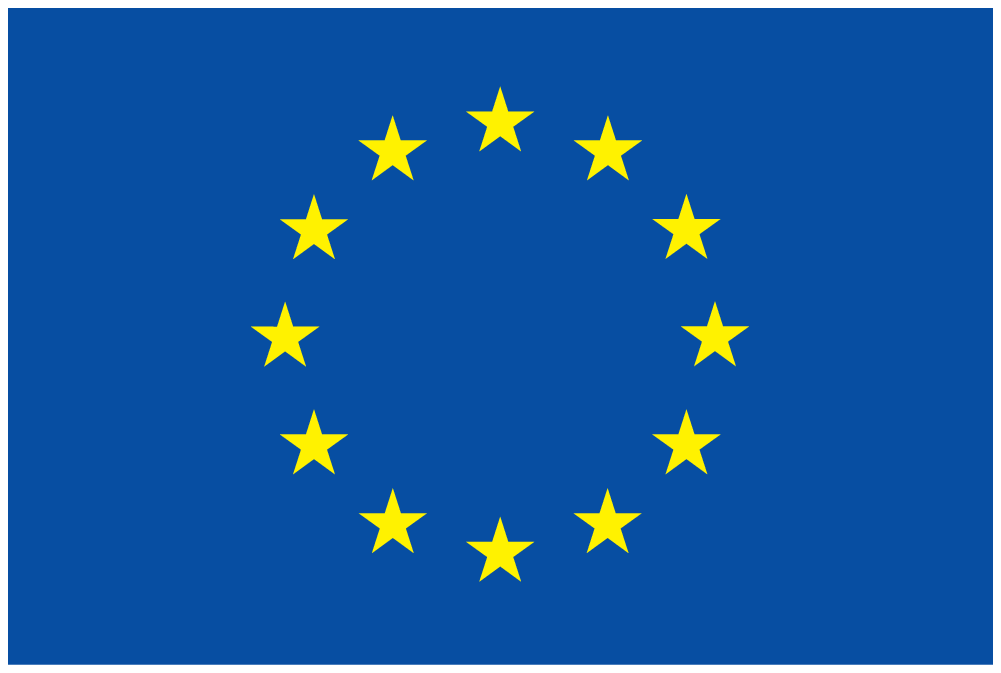} \end{minipage}  \hspace{-2cm} \begin{minipage}[l][1cm]{0.7\textwidth}
			This project has received funding from the European Union's Horizon 2020 research and innovation programme under the Marie Sk\l{}odowska-Curie grant agreement No 734922.
	\end{minipage}}

	\begin{abstract}
	An ordered set $S$ of vertices of a graph $G$ is a \emph{resolving set} for $G$ if every vertex is uniquely determined by its vector of distances to the vertices in $S$. The \emph{metric
	dimension} of G is the minimum cardinality of a resolving set. In this paper we study resolving sets tolerant to several failures in three-dimensional grids. 
Concretely, we seek for minimum cardinality sets that are resolving after removing any $k$ vertices from the set. 	
This is equivalent to finding \emph{$(k+1)$-resolving sets}, a generalization of resolving sets,
where, for every pair of vertices, the vector of distances to the vertices of the set differ in at least $k+1$ coordinates.
This problem  is also related with the study of the  \emph{$(k+1)$-metric dimension} of a graph, defined as the minimum cardinality of a $(k+1)$-resolving set.
In this work, we first prove that the metric dimension of a three-dimensional grid is 3 and establish some properties involving resolving sets in these graphs.
Secondly, we determine the values of $k\ge 1$ for which there exists a 
$(k+1)$-resolving set 
and construct such a resolving set of minimum cardinality in almost all cases.
\vspace{2mm}
	\end{abstract}
	\vspace{.01cm} \hspace{.95cm}{\it Key words:} Resolving set; metric dimension;  $k$-resolving set; 	$k$-metric dimension; fault-tolerant; three-dimensional grid.




\section{Introduction}

Resolving sets can be used to distinguish the vertices of a graph $G$ comparing distances to fixed vertices. Fault-tolerant resolving sets were defined to distinguish the vertices of $G$ even though one of the vertices of the set fails.
Here we consider a more general case, concretely, resolving sets that distinguish the vertices of a graph when any $k$ vertices of the set fail, where $k$ is a fixed integer.

The notion of resolving sets in graphs was defined independently by Harary and Melter~\cite{HM76} and 
Slater~\cite{S75} and have since received a lot of attention due to their applications
in several areas, such as network discovery and verification~\cite{BEEHHMR06}, robot navigation~\cite{KRR96}, chemistry~\cite{CEJO00} or games~\cite{C83}. 
Fault-tolerant resolving sets were introduced in \cite{HMSW} and have been studied in \cite{BGY19,CJS10,JSCS09,KR16,RHP19,SS15,SBMM18,V17}. 
For more applications and properties on metric dimension and its variants, the reader is addressed to the surveys  \cite{KY21,TFL21} and the references herein.

Let $G$ be a simple finite connected graph. For two vertices $u,v \in V(G)$, let $d(u,v)$ denote the length of a shortest path from $u$ to $v$. 
A vertex $u$ of $G$ \emph{resolves} two vertices $x$ and $y$ if $d(u,x)\not= d(u,y)$.
A set of vertices $S\subseteq V(G)$ is a \emph{resolving set} for $G$ if for every pair of different vertices $x$ and $y$ of $G$ there is a vertex $u$ in $S$ that resolves $x$ and $y$.
The \emph{metric dimension} of $G$, denoted by $\dim (G)$, is the minimum cardinality of a resolving set, and a \emph{metric basis} is a resolving set of cardinality $\dim (G)$.
If $S=\{ u_1,\dots ,u_k \}\subseteq V(G)$, we denote by $r(x|S)$ the vector of distances from $x$ to the vertices of $S$, that is, $r(x|S)=(d(x,u_1),\dots ,d(x,u_k))$.
Thus, $S$ is a resolving set if and only if $r(x|S)\not= r(y|S)$ for every pair of distinct vertices $x,y\in V(G)$. The elements of $r(x|S)$ are the \emph{metric coordinates} of $x$ with respect to $S$.
A resolving set $S\subseteq V(G)$ is \emph{fault-tolerant} if $S-\{ u\}$  is a resolving set, for every $u\in S$.

Resolving sets can be used to locate nodes in a network modeled as a graph, and fault-tolerant resolving sets 
can be used to locate nodes even though one of the nodes of the resolving set fails. 
Here we consider the possibility of more than one failure. 
Note that the set obtained after the removal of any $k$ vertices of a resolving set $S$ remains resolving if and only if every pair of vertices of the graph is resolved by at least $k+1$  distinct vertices of $S$. This last concept was introduced in \cite{EMRVY14}, concretely,
 a set $S$ of vertices of a graph is a \emph{$k$-resolving set} if for every pair of vertices $x$ and $y$,  the vectors $r(x|S)$ and $r(y|S)$ have at least $k$ different coordinates.
Hence, a set $S$ remains resolving  even though $k$ vertices fail if and only if $S$ is a $(k+1)$-resolving set.
The $k$-metric dimension and $k$-resolving sets of a graph and, concretely, of some product graphs have been studied in \cite{BGY19,EMRVY14,EMRVY15,EMRVY16cor, EMRVY16lex,KRT21,SVW21,YER-17}.
In the survey \cite{KY21}, an extensive summary of known results and applications of the $k$-metric dimension is given.

It is worth mentioning that, in contrast to resolving sets or fault-tolerant resolving sets, $k$-resolving sets do not always exist for $k\ge 3$ \cite{EMRVY14}.
Whenever a graph $G$ has at least one $k$-resolving set, a $k$-resolving set of minimum cardinality is a \emph{$k$-metric basis}, and its cardinality is the \emph{$k$-metric dimension} of $G$, denoted by $\dim_k(G)$  \cite{EMRVY14}. 
With this terminology,  $1$-resolving sets correspond to resolving sets and $2$-resolving sets correspond to  {fault-tolerant resolving sets}. 

Here we are interested in finding resolving sets tolerant to $k$ failures, that is, $(k+1)$-resolving sets and, more precisely, $(k+1)$-metric bases and the $(k+1)$-metric dimension of three-dimensional grids.
Grid graphs have been proven to be very useful in diverse areas such as
robotics, video games and telecommunications \cite{AAA17,DAAT17,KR16}.
The value of the $k$-metric dimension 
of two-dimensional grids is determined in \cite{BGY19}. 

Since a superset of a resolving set is also a resolving set, the following result  is obvious and provides a necessary and sufficient condition for the existence of a $k$-resolving set.

\begin{remark}\label{prop:nonexistencekftrs}
	A graph $G$ has a $k$-resolving set if and only if $V(G)$ is a $k$-resolving set.
\end{remark}

Next result gives a lower bound on the $k$-metric dimension of a graph, whenever it is defined.

\begin{prop}\label{prop:boundfacil} \cite{EMRVY14}
	If a graph $G$ has a $k$-resolving set, then $\dim_{k} (G)\ge \dim (G)+k-1$.
\end{prop}

{The paper is organised as follows. 
In Section~\ref{sec:resolvingvertices}, some properties concerning resolving sets for three-dimensional grids are given.  Concretely, we prove that the metric dimension of a three-dimensional grid is exactly 3 and establish lower bounds on the number of vertices resolving a fixed pair of vertices.  These bounds will be very useful in Section~\ref{sec:3Dgrids}, devoted to the study of resolving sets tolerant to $k$ failures in three-dimensional grids. Concretely, we determine the $(k+1)$-metric dimension  and describe $(k+1)$-metric bases of these grids in almost all cases.}




\section{Resolving sets in three-dimensional grids}\label{sec:resolvingvertices}

In this section we calculate the metric dimension of three-dimensional grids and prove some new results concerning resolving sets in these grids. 
We also point out sets of vertices resolving a fixed pair of vertices,  that will be  
very useful in Section~\ref{sec:3Dgrids}.

Formally, an \emph{$r$-dimensional grid}, or \emph{rD grid} for short, is any graph obtained as the cartesian product of $r$ non-trivial paths, that is, $P_{n_1}\Box \cdots \Box P_{n_r}$, with $n_1,\dots, n_r\ge 2$.
We can assume that the set of vertices of $P_{n_1}\Box \cdots \Box P_{n_r}$  is: 

\vspace{-3mm}\begin{equation*}
V(P_{n_1}\Box \cdots \Box P_{n_r})=\{ (x_1,\dots ,x_r) : 0\le  x_i \le n_i-1 \hbox{ for every }i\in \{1,\dots ,r\} \}
\end{equation*}
and two vertices $(x_1,\dots ,x_r)$ and $(y_1,\dots ,y_r)$ of $V(P_{n_1}\Box \cdots \Box P_{n_r})$ are adjacent if and only if $y_i\in \{x_i-1,x_i+1\}$, for some $i\in \{ 1,\dots, r\}$, and $x_j=y_j$, whenever $j\not= i$. 
Hence, the degrees of the vertices  of an $r$D grid are $r, r+1,\dots,2r$, and the distance between two vertices $(x_1,\dots ,x_r)$ and $(y_1,\dots ,y_r)$ in $P_{n_1}\Box \cdots \Box P_{n_r}$ is:

\vspace{-3mm}\begin{equation*}
d( (x_1,\dots ,x_r), (y_1,\dots ,y_r))=\sum_{i=1}^r |y_i-x_i|.
\end{equation*}

There are some vertices that play an important role  in resolving pairs of vertices of a $3$D grid, concretely, vertices with degree different from the maximum degree, 6. 
The vertices of degree $3$ are called \emph{corners}. 
A vertex of degree 6 is an \emph{interior} vertex.
With the specified labeling of the vertices, the set of corners of a 3$D$ grid is  $\{ (x_1,x_2,x_3): x_i\in \{ 0, n_i-1\} \hbox{ for }i\in \{1,2,3 \} \}$. 
A \emph{face} of a \threedim grid  is a set of vertices with a constant coordinate $x_{i_0}$ equal to either $0$ or to $n_{i_0}-1$, for some $i_0\in\{1,2,3\}$.
Let $F(n_1,n_2,n_3)$ denote the set of vertices belonging to any face, that is,  $F(n_1,n_2,n_3)=\{(x_1,x_2,x_3): \, x_i\in \{ 0, n_i-1\} \,\,\hbox{ for some }i\in \{1,2,3\} \,\}$. We write simply $F$ if the values $n_1$, $n_2$ and $n_3$ are clear from context.
Note that $F$ consists of all non-interior vertices and
$ |F(n_1,n_2,n_3)|=n_1n_2n_3-(n_1-2)(n_2-2)(n_3-2)=2(n_1n_2+n_2n_3+n_1n_3)-4(n_1+n_2+n_3)+8.$

It is known that the metric dimension of an $r$D grid is at most $r$~\cite{KRR96}.
Also in this paper, it is said that it is exactly $r$ and the proof is left to the reader. However, in general this is not in true. For example, it is known that the metric dimension of the hypercube, that can be viewed as a $r$D grid with $n_1=\dots =n_r=2$, is less than $r$ for $r\ge 5$ \cite{CHMPPSW07}.
It remains an open problem to determine the exact value of the metric dimension for general grids.
A discussion on asymptotic values of the metric dimension of a grid  is included in \cite{JP19}.
Next, we prove that the result stated in~\cite{KRR96} holds for 3D grids.

%
%

\begin{theorem}\label{thm:basisGrid3D}
	If $n_1,n_2, n_3\ge 2$, then $\dim (P_{n_1}\Box P_{n_2}\Box P_{n_3})=3$. 
	Moreover, a set formed by three corners of a face is a metric basis.	
\end{theorem}
\begin{proof} We begin by proving that $\dim (P_{n_1}\Box P_{n_2}\Box P_{n_3})\ge 3$. 
  Paths are the only graphs with metric dimension 1~\cite{CEJO00}, thus, it is enough to prove that the metric dimension of $P_{n_1}\Box P_{n_2}\Box P_{n_3}$ is different from $2$. 

 Suppose to the contrary that $\dim (P_{n_1}\Box P_{n_2}\Box P_{n_3})=2$ and $S=\{ u,v \}$ is a resolving set for $G$. Then, the degree of  $u$ and $v$ is at most 3 \cite{KRR96}. Hence, $u$ and $v$ must be corners in $P_{n_1}\Box P_{n_2}\Box P_{n_3}$. We can assume without loss of generality that $u=(0,0,0)$ and $v\in \{ (n_1-1,0,0), (n_1-1,n_2-1,0),  (n_1-1,n_2-1,n_3-1) \}$.
 But, if $v=(n_1-1,0,0)$, then $r((0,1,0)|S)=r((0,0,1)|S)=(1,n_1)$;
 if $v= (n_1-1,n_2-1,0)$, then $r((0,1,0)|S)=r((1,0,0)|S)=(1,n_1+n_2-3)$; and 
 if $v= (n_1-1,n_2-1,n_3-1)$, then $r((0,1,0)|S)=r((1,0,0)|S)=(1,n_1+n_2+n_3-4)$, 
 which is a contradiction.
  
 In \cite{KRR96}, the authors prove that the set $\{ (0,0,0), (n_1-1,0,0), (0,n_2-1,0) \}$ is a resolving set for $P_{n_1}\Box P_{n_2}\Box P_{n_3}$. Hence, $\dim (P_{n_1}\Box P_{n_2}\Box P_{n_3})= 3$. By symmetry, we have that every set formed by three corners of a face is a metric basis.
\end{proof}

The following lemma describes some resolving sets of cardinality 4 for a \threedim grid. 
Notice that similar resolving sets can be given by interchanging the role of the three coordinates.

\begin{lemma}\label{lem:grid3resset4}
	If $h,h'\in \{ 0,\dots , n_3-1 \}$, with $h\not=h'$, $i\in \{ 0, \dots ,n_1-1 \} $ and $j\in \{ 0, \dots ,n_2-1 \} $, then  $S=\{(0,0,h), (n_1-1,0,h), (0,n_2-1,h), (i,j,h')\}$  is a resolving set for $P_{n_1}\Box P_{n_2}\Box P_{n_3}$.
\end{lemma}
\begin{proof} Let $G=P_{n_1}\Box P_{n_2}\Box P_{n_3}$. By  Theorem~\ref{thm:basisGrid3D}, the set $S'=\{(0,0,h), (n_1-1,0,h), (0,n_2-1,h)\}$ is a metric basis for the \threedim grid formed by the vertices $\{ (x_1,x_2,x_3)\in V(G): x_3\le h\}$ and for the \threedim grid formed by the vertices $\{ (x_1,x_2,x_3)\in V(G): x_3\ge h\}$. Hence, if $h\in \{0,n_3-1\}$, then $S$ is a resolving set for $G$. Otherwise, by symmetry, two different vertices $u$ and $v$ of $G$ have the same coordinates with respect to $S$
if and only if $u=(a,b,c_1)$ and $v=(a,b,c_2)$, with $c_1\not=c_2$ and $c_1+c_2=2h$, for some integers $a,b,c_1,c_2$. 
We claim that $u$ and $v$ are resolved by $(i,j,h')$.
Indeed, suppose to the contrary that $(i,j,h')$ does not resolve $u$ and $v$. 
In such a case, since
$d(u,(i,j,h'))=|i-a|+|j-b|+|h'-c_1|$ and $d(v,(i,j,h'))=|i-a|+|j-b|+|h'-c_2|$,
we have $|h'-c_1|=|h'-c_2|$. But this equality is true if and only if $c_1=c_2$ or $c_1+c_2=2h'$.
Hence, $h=h'$, a contradiction.
\end{proof}
Now, we give a sufficient condition for a set not to be resolving in a \threedim grid. First we need some terminology. For every $a_1\in \{1,\dots ,n_1-1\}$ and $a_2\in \{1,\dots ,n_2-1\}$, we define the sets of vertices in $P_{n_1}\Box P_{n_2}$:

\vspace{-4mm}
\begin{align*}
R_{--}(a_1,a_2)&=\{ (x_1,x_2): 0\le x_1< a_1, 0\le x_2< a_2 \},\\
R_{++}(a_1,a_2)&=\{ (x_1,x_2): a_1\le x_1<n_1, a_2\le x_2<n_2 \},\\
R_{-+}(a_1,a_2)&=\{ (x_1,x_2): 0\le x_1< a_1, a_2\le x_2<n_2 \},\\
R_{+-}(a_1,a_2)&=\{ (x_1,x_2): a_1\le x_1<n_1, 0\le x_2< a_2 \}.
\end{align*} 
Notice that  these sets form a partition of the set of vertices of the $2$D grid $P_{n_1}\Box P_{n_2}$
 (see Figure~\ref{fig:regiones}).

\begin{figure}[t!]
	\centering
	\includegraphics[width=0.5\textwidth]{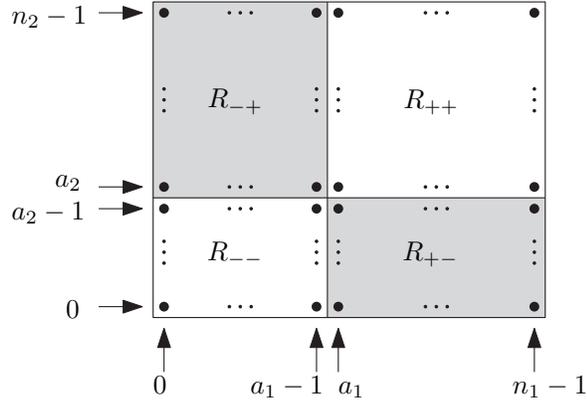}
	\caption{The sets $R_{--}(a_1,a_2)$, $R_{++}(a_1,a_2)$,
	$R_{-+}(a_1,a_2)$ and $R_{+-}(a_1,a_2)$. 
}
	\label{fig:regiones}
\end{figure}

Let $\textrm{pr}_3(S)$ be the projection of $S$ onto $P_{n_1}\Box P_{n_2}$, that is, 
$\textrm{pr}_3(S)$ contains all the vertices of the $2$D grid $P_{n_1}\Box P_{n_2}$ 
obtained by deleting the third coordinate from the vertices of $S$. 
Analogously, 
the projection $\textrm{pr}_{i}(S)$, $i\in \{ 1,2 \}$,  
contains all the vertices of the $2$D grid obtained by deleting the $i$-th coordinate from the vertices of $S$.

\begin{lemma}\label{lem:regiones3D}
	If there exist $a_1\in \{1,\dots ,n_1-1\}$ and $a_2\in \{1,\dots ,n_2-1\}$ such that
		$\textrm{pr}_3(S)\subseteq R_{--}(a_1,a_2)\cup R_{++}(a_1,a_2)$ or $\textrm{pr}_3(S)\subseteq R_{-+}(a_1,a_2)\cup R_{+-}(a_1,a_2)$,
        then $S$ is not a resolving set for $P_{n_1}\Box P_{n_2}\Box P_{n_3}$.
\end{lemma}
\begin{proof}
	If $\textrm{pr}_3(S)\subseteq R_{--}(a_1,a_2)\cup R_{++}(a_1,a_2)$, then the vertices $(a_1-1,a_2,0)$  and $(a_1, a_2-1,0)$ are not resolved by $S$.  Anagolously, if $\textrm{pr}_3(S)\subseteq R_{-+}(a_1,a_2)\cup R_{+-}(a_1,a_2)$, then the vertices $(a_1-1,a_2-1,0)$  and $(a_1, a_2,0)$ are not resolved by $S$ (see Figure~\ref{fig:regiones3D}). 
\end{proof}

Note that, by symmetry, the preceding result can also be stated  for $\textrm{pr}_{1}(S)$  and $\textrm{pr}_{2}(S)$.

\begin{figure}[t!]
	\centering
	\includegraphics[width=0.6\textwidth]{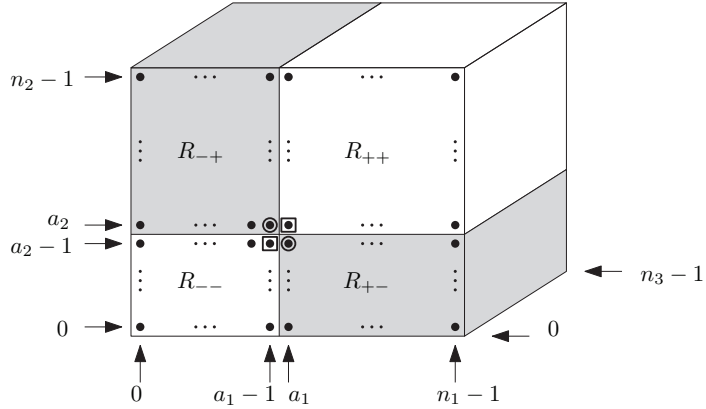}
	\caption{If $\textrm{pr}_3(S)\subseteq R_{--}(a_1,a_2)\cup R_{++}(a_1,a_2)$, then $S$ is included in the white region and circled vertices are not resolved by any vertex of $S$. Analogously, if
		$\textrm{pr}_3(S)\subseteq R_{-+}(a_1,a_2)\cup R_{+-}(a_1,a_2)$, then $S$ is included in the gray region and squared vertices are not resolved by any vertex of $S$.}
	\label{fig:regiones3D}
\end{figure}

Now, we prove some properties for general graphs that provide vertices resolving a fixed pair of vertices under certain conditions.

\begin{lemma}\label{prop:shortestpath}
{Let $u$ and $v$ be vertices of a graph $G$. If $d(u,v)$ is odd, then every vertex of a shortest path between $u$ and $v$ resolves $u$ and $v$. If $d(u,v)$ is even, then all but one of the vertices of a shortest path between $u$ and $v$ resolve $u$ and $v$. }
\end{lemma}
\begin{proof}
	For every vertex $x$ belonging to a shortest path from $u$ to $v$, the equality $d(u,x)+d(x,v)=d(u,v)$ holds. If $d(u,v)$ is odd, then there is no vertex $x$ in the shortest path satisfying $d(u,x)=d(v,x)$.
	If $d(u,v)$ is even, then there is exactly one vertex $x$ in the shortest path satisfying $d(u,x)=d(v,x)=d(u,v)/2$.
\end{proof}

\begin{lemma}\label{lem:resolvingbipartite}
	Let $u$ and $v$ be vertices of a bipartite graph $G$.
	If $d(u,v)$ is odd, then every vertex of $G$ resolves $u$ and $v$.
\end{lemma}
\begin{proof}
	The distance between two vertices of the same partite set of $G$ is even, and the distance 
	between vertices of different partite sets is odd.
	If $d(u,v)$ is odd, then $u$ and $v$ belong to different partite sets of $G$, and  
	for every vertex $x\in V(G)$, the distances $d(x,u)$ and $d(x,v)$ have different parity. 
	Hence, $x$ resolves $u$ and $v$.
\end{proof}

Next results establish a lower bound on the number of vertices in $F$ resolving a pair of fixed vertices $u$ and $v$, taking into account the number of different coordinates of $u$ and $v$. 

Since grids are bipartite graphs, by Lemma~\ref{lem:resolvingbipartite}, all the vertices of a grid resolve a pair of vertices $u$ and $v$, if $d(u,v)$ is odd.
Thus, we focus on the case $d(u,v)$ even. For $n_1,n_2,n_3\ge 2$, we define

\vspace{-5mm}
\begin{align*}
	\alpha_M(n_1,n_2,n_3)&=\min \{ n_1(n_2+n_3-2),n_2(n_1+n_3-2) ,n_3(n_1+n_2-2) \}.
\end{align*}
It is easy to check that $\alpha_M(n_1,n_2,n_3)=n_i(n_1+n_2+n_3-n_i-2)$, for $n_i=\min \{ n_1,n_2,n_3 \}$.

\begin{lemma}\label{lem:resolvingrecta} Let $u=(x_1,x_2,x_3)$ and $v=(y_1,y_2,y_3)$ be two different vertices of $P_{n_1}\Box P_{n_2}\Box P_{n_3}$ with $x_i\not=y_i$ for exactly one value $i\in \{1,2,3\}$. Then, there are at least $\alpha_M(n_1,n_2,n_3)$ vertices in $F(n_1,n_2,n_3)$ resolving $u$ and $v$.
\end{lemma}
\begin{proof} Assume without loss of generality that $x_1=y_1$, $x_2=y_2$ and $x_3\not=y_3$. 
A vertex $z=(z_1,z_2,z_3)$ does not resolve $u$ and $v$ if and only $d(z,u)=|x_1-z_1|+|x_2-z_2|+|x_3-z_3|=|x_1-z_1|+|x_2-z_2|+|y_3-z_3|=d(z,v)$, that is, if and only if $|x_3-z_3|=|y_3-z_3|$.
Hence, the set of vertices not resolving $u$ and $v$ is $\{ (a,b,(x_3+y_3)/2) : 0\le a\le n_1-1, 0\le b\le n_2-1\}$ (note that this set is empty when $x_3$ and $y_3$ have different parity). In any case, there are at least $|F|- (2n_1+2n_2-4)$ vertices in $F$ resolving $u$ and $v$. 
But	

\vspace{-5mm}
\begin{align*}
|F|&- (2n_1+2n_2-4)-n_3(n_1+n_2-2)\\
=&2(n_1-2)(n_2-2)+(n_2-2)(n_3-2)+(n_1-2)(n_3-2)+2(n_3-2)\ge 0.
\end{align*}
Hence, $|F|- (2n_1+2n_2-4)\ge n_3(n_1+n_2-2)\ge \alpha_M$.
\end{proof}

Lemma~\ref{prop:shortestpath}  shows when a vertex of a shortest path between two fixed vertices $u$ and $v$ resolves the pair $u$ and $v$.  Next, we analize which vertices not belonging to the shortest path resolve $u$ and $v$.  We denote by $S(u,v)$ the set of vertices belonging to a shortest path between $u$ and $v$ in a 3D grid.
Concretely, for every pair $u=(x_1,x_2,x_3)$ and $v=(y_1,y_2,y_3)$ of distinct vertices of the grid $P_{n_1}\Box P_{n_2}\Box P_{n_3}$,   $S(u,v)$ is the set of vertices: 

\vspace{-3mm}
{\small
$$[\min\{x_1,y_1\}, \max \{x_1,y_1\}]\times [\min\{x_2,y_2\}, \max \{x_2,y_2\}]\times [\min\{x_3,y_3\}, \max \{x_3,y_3\}].$$ }
Hence, $S(u,v)$ induces a 3D grid such that $u$ and $v$ are corners, if the three coordinates of $u$ and $v$ are different; a 2D grid, if  $u$ and $v$ have exactly 2 different coordinates; and a path,  if  $u$ and $v$ differ by exactly one coordinate.

Next, we  associate a vertex not in $S(u,v)$  with a vertex of the subgrid $S(u,v)$ (belonging to one of its faces, whenever $S(u,v)$ is a 3D grid) so that either both vertices resolve $u$ and $v$ or neither of them resolves $u$ and $v$. 
The result is stated for the first coordinate, but it holds similarly for the second and third coordinates.

\begin{lemma}\label{lem:alsoresolving} Let $u=(x_1,x_2,x_3)$ and $v=(y_1,y_2,y_3)$ be two distinct vertices of $P_{n_1}\Box P_{n_2}\Box P_{n_3}$. Let $z=(z_1,z_2,z_3)$.
	\begin{enumerate}[i)]
		\item If $z_1\le \min\{x_1,y_1\}$, then $z$ resolves $u$ and $v$ if and only if 
		vertex $(\min\{x_1,y_1\},z_2,z_3)$ 
		resolves $u$ and $v$.
		\item If $z_1\ge \max\{x_1,y_1\}$, then $z$ resolves $u$ and $v$ if and only if 
	    vertex $(\max\{x_1,y_1\},z_2,z_3)$ 
		resolves $u$ and $v$.
	\end{enumerate}
\end{lemma}
\begin{proof} We prove only the first item, because the second one is derived analogously. Assume without loss of generality that $x_1\le y_1$. 
	Let $z=(z_1,z_2,z_3)$ be such that $z_1\le x_1$ and let $z'=(x_1,z_2,z_3)$. 
	Then, $d(z,u)=|x_1-z_1|+|x_2-z_2|+|x_3-z_3|=d(z,z')+d(z',u)$.
	Similarly, we have that  $d(z,v)=|y_1-z_1|+|y_2-z_2|+|y_3-z_3|=|y_1-x_1|+|x_1-z_1|+|y_2-z_2|+|y_3-z_3|
	=|x_1-z_1|+(|y_1-x_1|+|y_2-z_2|+|y_3-z_3|)=d(z,z')+d(z',v)$.
	Hence, $z$ resolves $u$ and $v$ if and only if $z'$ resolves $u$ and $v$.
\end{proof}

Note that the preceding lemma implies that a vertex $z$ not in $S(u,v)$ resolves $u$ and $v$ if and only if the vertex 
in $S(u,v)$ closest to $z$ resolves $u$ and $v$.

\begin{lemma}\label{lem:resolvingpla} Let $u=(x_1,x_2,x_3)$ and $v=(y_1,y_2,y_3)$ be two different vertices of $P_{n_1}\Box P_{n_2}\Box P_{n_3}$ with $x_i=y_i$ for exactly one value of $\{1,2,3\}$. Then, there are at least $\alpha_M(n_1,n_2,n_3)$ vertices of $F(n_1,n_2,n_3)$ resolving $u$ and $v$.
\end{lemma}
\begin{proof} 
We may assume without loss of generality that $x_1<y_1$, $x_2<y_2$ and $x_3=y_3$. Then, $d(u,v)=y_1-x_1+y_2-x_2$. 
If $d(u,v)$ is odd, then all the vertices in $F$ resolve $u$ and $v$ by Lemma~\ref{lem:resolvingbipartite}.
Now suppose that $d(u,v)=y_1-x_1+y_2-x_2$ is even.
If $y_1-x_1=y_2-x_2$, then, using Lemma~\ref{lem:alsoresolving} we derive that the vertices in $F$ belonging to the set 

\vspace{-5mm}
 \begin{align*}
&\{ (0,b,c) : 0\le b<y_2, 0\le c\le n_3-1\}\\ &\cup \{ (n_1-1,b,c) : x_2< b<n_2-1, 0\le c\le n_3-1\}\\ &\cup 
\{ (a,0,c) : 0\le a<y_1, 0\le c\le n_3-1\}\\ &\cup \{ (a,n_2-1,c) : x_1< a<n_1-1, 0\le c\le n_3-1\}
 \end{align*}
resolve $u$ and $v$, and the number of vertices of this set is $(y_2+(n_2-1-x_2)+y_1+(n_1-1-x_1)-2)n_3=(y_2-x_2+y_1-x_1+n_1+n_2-4)n_3\ge 
(n_1+n_2-2)n_3$ (see Figure~\ref{fig:nonresolvingcapa}, left).

\begin{figure}[t!]
	\centering
	\includegraphics[width=0.9\textwidth]{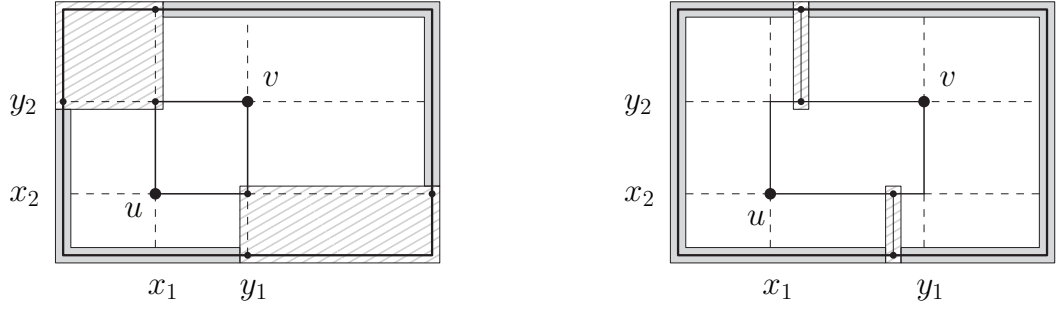}
	\caption{Vertices $(a,b,c)$, with $c=x_3$, of the 3D grid are represented.  Left, case $y_1-x_1=y_2-x_2$ and right, $y_1-x_1\not= y_2-x_2$. Vertices  in the striped region do not resolve $u$ and $v$. In both cases, there are at least $n_1+n_2-2$ vertices in the boundary resolving $u$ and $v$ (gray region).
	}
	\label{fig:nonresolvingcapa}
\end{figure}

If  $y_1-x_1\not= y_2-x_2$, we may assume $y_1-x_1 < y_2-x_2$. Let $r=(y_1-x_1+y_2-x_2)/2$. Then,
the following vertices in $F$ resolve $u$ and $v$:

\vspace{-5mm}
\begin{align*}
&\{ (0,b,c) : b\not= x_2+r,  0\le c\le n_3-1\}\\ 
&\cup \{ (n_1-1,b,c) : b\not= y_2-r,  0\le c\le n_3-1\}\\ 
&\cup \{ (a,0,c) : 0\le a\le n_1-1, 0\le c\le n_3-1\}\\ 
&\cup \{ (a,n_2-1,c) : 0\le a\le n_1-1, 0\le c\le n_3-1\}.
\end{align*}
and this set has at least $(2n_1+2n_2-4)n_3$ vertices.
Hence, there are at least $(n_1+n_2-2)n_3\ge \alpha_M$ vertices in $F$ resolving $u$ and $v$  (see Figure~\ref{fig:nonresolvingcapa}, right).
\end{proof}

\begin{lemma}\label{lem:resolving2capas} Let $u=(x_1,x_2,0)$ and $v=(y_1,y_2,n_3-1)$ be two vertices of $P_{n_1}\Box P_{n_2}\Box P_{n_3}$. Then, there are at least $n_3(n_1+n_2-2)$ vertices in $F(n_1,n_2,n_3)$ resolving $u$ and $v$.
\end{lemma}
\begin{proof}
	We claim that for every pair $a$ and $b$, such that $a\in \{ 0,n_1-1\}$ or $b\in \{ 0, n_2-1\}$, there is at most one vertex in the set $T_{a,b}=\{ (a,b,c): 0\le c\le n_3-1 \}$ not resolving $u$ and $v$.
	Indeed, if there were at most two vertices not resolving them, then by Lemma~\ref{lem:alsoresolving} there would be two vertices $z$ and $z'$ belonging to  
	some set $T'_{a',b'}=\{ (a',b',c): 0\le c\le n_3-1  \}$ not resolving $u$ and $v$, where $x_1\le a'\le y_1$, $x_2\le b'\le y_2$ and either $a'\in\{ x_1,y_1\}$ or $b'\in \{x_2,y_2\}$. But this is a contradiction by Lemma~\ref{prop:shortestpath}, because there is a shortest path between $u$ and $v$ that goes through both $z$ and $z'$.
 Notice that there are $2n_1+2n_2-4$ sets of type $T_{a,b}$, with $a\in \{ 0,n_1-1\}$ or $b\in \{ 0, n_2-1\}$. Hence, there are at least $(2n_1+2n_2-4)(n_3-1)$ vertices in $F(n_1,n_2,n_3)$ resolving $u$ and $v$, and  
 $(2n_1+2n_2-4)(n_3-1)\ge (n_1+n_2-2)(2n_3-2)\ge (n_1+n_2-2)n_3$, where the last inequality holds because $n_3\ge 2$.
\end{proof}

\begin{lemma}\label{lem:resolvingcapa} Let $u=(x_1,x_2,x_3)$ and $v=(y_1,y_2,y_3)$ be two different vertices of $P_{n_1}\Box P_{n_2}\Box P_{n_3}$ with $x_i\not=y_i$, for every $i\in \{ 1,2,3 \}$. Then, there are at least $n_1+n_2-2$ vertices $(z_1,z_2,z_3)$ resolving $u$ and $v$ with $z_3=n_3-1$, and either $z_1\in \{0,n_1-1\}$ or $z_2\in \{0,n_2-1\}$.
\end{lemma}
\begin{proof} 
	Assume without loss of generality that $x_i<y_i$ for all $i\in \{1,2,3\}$. Let $C=\{(z_1,z_2,n_3-1): z_1\in \{0,n_1-1\}\hbox{ or }z_2\in \{0,n_2-1\}\}$ and let $C'=\{(z_1,z_2,y_3): z_1\in \{0,n_1-1\}\hbox{ or }z_2\in \{0,n_2-1\}\}$. Note that $|C|=|C'|=2n_1+2n_2-4$. 
	By Lemma~\ref{lem:alsoresolving}, a vertex $z=(z_1,z_2,n_3-1)\in C$ resolves $u$ and $v$ if and only if 
	$(z_1,z_2,y_3)\in C'$ resolves $u$ and $v$. Therefore, it is enough to show that there are at least $n_1+n_2-2$ vertices in $C'$ resolving $u$ and $v$.
	
	Consider the sets $S_1=\{ (x_1,b, y_3): x_2\le b\le y_2 \}\cup \{ (a,y_2, y_3): x_1< a\le y_1 \}$,  
	$S_2=\{ (a,x_2, y_3): x_1\le a\le y_1 \}\cup \{ (y_1,b, y_3): x_2< b\le y_2 \}$ and $S_3=\{ (x_1,x_2,c): x_3\le c < y_3\}$. 
	The vertices of $S_1\cup S_3$ form a shortest path between $u$ and $v$ as well as the vertices of $S_2\cup S_3$. 
    Hence, there is at most one vertex in $S_1$ and at most one vertex in $S_2$ not resolving $u$ and $v$. 
    Let $S$ be the set of vertices in $S_1\cup S_2$ not resolving $u$ and $v$. 
    Thus, $|S|\in \{ 0,1,2\}$.
    Note that for $w=(z_1,z_2,y_3)\in S$, we have $d(w,v)=d(w,u)= d(w, (x_1,x_2,y_3))+(y_3-x_3)>d(w, (x_1,x_2,y_3))$
    
    If $|S|=0$, then all the vertices of $C'$ resolve $u$ and $v$ by Lemma~\ref{lem:alsoresolving}, and  $|C'|=2n_1+2n_2-4> n_1+n_2-2$.
     
    If $|S|=1$, then, since $|S_1|=|S_2|$, the vertex in $S$ must be $(x_1,x_2,y_3)$. By Lemma~\ref{lem:alsoresolving}, there are 
    $2n_1+2n_2-4-(x_1+x_2+1)$ vertices in $C'$ resolving $u$ and $v$.
    But $2n_1+2n_2-4-(x_1+x_2+1)\ge n_1+n_2-2$, because $x_i<y_i\le n_i-1$, for $i=1,2$ (see Figure~\ref{fig:nonresolvingcapalema11}(a)). 
    
    \begin{figure}[t!]
    	\centering
    	\includegraphics[width=0.7\textwidth]{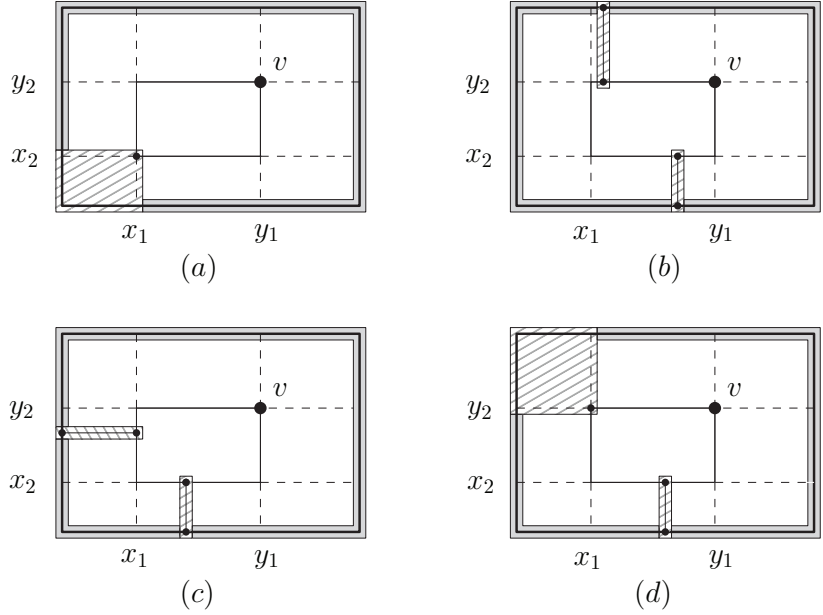}
    	\caption{Vertices $(a,b,c)$, with $c=y_3$, of the 3D grid are depicted. 
    		Vertices in the striped region do not resolve $u$ and $v$.
    		The vertices of $C'$ resolving $u$ and $v$ are in the gray region. 
    		(a) $|S|=1$. (b), (c) and (d), $|S|=2$.   		
    	}
    	\label{fig:nonresolvingcapalema11}
    \end{figure}    

    If $|S|=2$, then $(x_1,x_2,y_3)\notin S$,  one vertex of $S$ is in $S_1\setminus S_2$, the other one belongs to $S_2\setminus S_1$, and both are at the same distance from $(x_1,x_2,y_3)$.
    Moreover, there is at most one vertex belonging to $S$ in $\{(x_1,y_2,y_3), (y_1,x_2,y_3)\}$. Indeed, 
    if $S=\{(x_1,y_2,y_3), (y_1,x_2,y_3)\}$, then $y_1-x_1=y_2-x_2+y_3-x_3$ and $y_2-x_2=y_1-x_1+y_3-x_3$, because the vertices of $S$ are at the same distance from $u$ and $v$, and this is not possible because $y_3-x_3>0$. 
    If $S\cap \{(x_1,y_2,y_3), (y_1,x_2,y_3)\}=\emptyset$, then by Lemma~\ref{lem:alsoresolving} there are at least 
    $2n_1+2n_2-4-2$ vertices in $C'$ resolving $u$ and $v$, and $2n_1+2n_2-4-2\ge n_1+n_2-2$ because $n_1,n_2\ge 2$
    (see Figure~\ref{fig:nonresolvingcapalema11}(b,c)).
    If $|S\cap \{(x_1,y_2,y_3), (y_1,x_2,y_3)\}|=1$, we may assume without loss of generality that $y_2-x_2<y_1-x_1$ so that 
    $S\cap \{(x_1,y_2,y_3), (y_1,x_2,y_3)\}=(x_1,y_2,y_3)$.
    By Lemma~\ref{lem:alsoresolving} there are at least 
    $2n_1+2n_2-4-1-(x_1+n_2-y_2)$ vertices in $C'$ resolving $u$ and $v$,
    $2n_1+2n_2-4-1-(x_1+n_2-y_2)\ge n_1+n_2-2$ because $x_1<y_1\le n_1-1$ and $0\le x_2< y_2$
    (see Figure~\ref{fig:nonresolvingcapalema11}(d)).

   In any case, there are at least $n_1+n_2-2$ vertices in $C'$ resolving $u$ and $v$,  and so in $C$, as we wanted to prove.
\end{proof}

We finish this section with a lower bound on the number of vertices belonging to the faces of a 3D grid resolving any fixed pair of vertices.
\begin{prop}\label{prop:kftrsgap} 	 
	Let $u,v\in V(P_{n_1}\Box P_{n_2}\Box P_{n_3})$.
	Then, there are at least $\alpha_M(n_1,n_2,n_3)$ vertices in $F(n_1,n_2,n_3)$ resolving $u$ and $v$.
\end{prop}
\begin{proof}
	We proceed by induction on $h=n_1+n_2+n_3\ge 6$.
	
	If $h=6$, then $n_1=n_2=n_3=2$ and $\alpha_M(n_1,n_2,n_3)=4$. Besides, all the vertices of the grid are in $F(2,2,2)$ and it is easy to check that, for every pair of vertices $u$ and $v$, there are at least 4 vertices resolving them.
	
	Now we prove that the statement of the lemma holds whenever $n_1+n_2+n_3=h$, assuming that it is true for grids $P_{n_1'}\Box P_{n_2'}\Box P_{n_3'}$ such that $n_1'+n_2'+n_3'\le h-1$. 
	Let $u$ and $v$ be two different vertices of $P_{n_1}\Box P_{n_2}\Box P_{n_3}$ and 
	let $R_{n_1,n_2,n_3}(u,v)$ denote the set of vertices of $F(n_1,n_2,n_3)$ resolving $u$ and $v$.
	If $u$ and $v$ have at least one equal coordinate, then $|R_{n_1,n_2,n_3}(u,v)|\ge \alpha_M(n_1,n_2,n_3)$ by Lemmas~\ref{lem:resolvingrecta} and ~\ref{lem:resolvingpla}.
	Now suppose that the three coordinates of $u$ and $v$ are distinct.
	We may assume $n_1\le n_2\le n_3$.
	Suppose first that  $n_1<n_3$. 
	In such a case, $n_1+n_2+n_3-1=h-1$. If $u,v$ are in $P_{n_1}\Box P_{n_2}\Box P_{n_3-1}$, then
	by induction hypothesis and using Lemmas~\ref{lem:alsoresolving} and \ref{lem:resolvingcapa}, 
	
	\vspace{-5mm}
	\begin{align*}
	|R_{n_1,n_2,n_3}(u,v)|&\ge |R_{n_1,n_2,n_3-1}(u,v)| + (n_1+n_2-2)\\&\ge \alpha_M(n_1,n_2,n_3-1) +(n_1+n_2-2)\\&=(n_2+(n_3-1)-2)n_1+(n_1+n_2-2)\\&=(n_2+n_3-2)n_1+n_2-2\\&\ge (n_2+n_3-2)n_1=\alpha_M(n_1,n_2,n_3).
	\end{align*}
	If no both vertices $u$ and $v$ are included in a grid $P_{n_1}\Box P_{n_2}\Box P_{n_3-1}$, by symmetry it is enough to consider the case
	$u=(x_1,x_2,0)$ and $v=(y_1,y_2,n_3-1)$.
	By Lemma~\ref{lem:resolving2capas}, 
	$|R_{n_1,n_2,n_3}(u,v)|\ge n_3(n_1+n_2-2)\ge \alpha_M(n_1,n_2,n_3)$.
	
	Now suppose $n_1=n_3$, that is, $n_1=n_2=n_3=n$.
	Arguing as in the preceding case, if $u,v$ belong to  $P_{n}\Box P_{n}\Box P_{n-1}$, then
	
	\vspace{-5mm}
	\begin{align*}
	|R_{n,n,n}(u,v)|&\ge |R_{n,n,n-1}(u,v)| + (2n-2)\\&\ge \alpha_M(n,n,n-1) +(2n-2)\\&=(2n-2)(n-1)+(2n-2)\\&=(2n-2)n=\alpha_M(n,n,n)
	\end{align*}
	Otherwise, assume that $u=(x_1,x_2,0)$ and $v=(y_1,y_2,n-1)$,
	and by Lemma~\ref{lem:resolving2capas}, 
	$|R_{n,n,n}(u,v)|\ge n(2n-2)= \alpha_M(n,n,n)$. 	
\end{proof}




\section{\kfftrss{} in 3D grids}\label{sec:3Dgrids}

Recall that a resolving set $S$ of a graph remains resolving after the deletion of any $k$ vertices if and only if $S$ is a $(k+1)$-resolving set of $G$, and the minimum cardinality of such a set is the $(k+1)$-metric dimension of $G$.
 In this section we provide exact values and bounds on the  \kftmd{} of 3D grids.
Moreover,  in almost all cases, a $(k+1)$-metric basis, that is, a $(k+1)$-resolving set of minimum cardinality, is constructed. 

Let $\alpha_m(n_1,n_2,n_3)$
be defined as follows:
	
\vspace{-5mm}
\begin{align*}
\alpha_m(n_1,n_2,n_3)&=2(n_1+n_2+n_3)-8, \nonumber 
\end{align*}
and recall the definition of $\alpha_M(n_1,n_2,n_3)$ given in the preceding section:
	
\vspace{-5mm}
\begin{align*}
	\alpha_M(n_1,n_2,n_3)&=\min \{ n_1(n_2+n_3-2),n_2(n_1+n_3-2) ,n_3(n_1+n_2-2) \}.
\end{align*} 
From now on, we assume $n_1,n_2,n_3\ge 2$ and $k\ge 2$.
Note that $\alpha_m$ and $\alpha_M$ are half the number of vertices lying on an ``edge'' of the grid and half the number of vertices lying on the four smallest faces of the grid, respectively (see Figure~\ref{fig:alphas}).

\begin{figure}[t!]
	\centering
	\includegraphics[width=0.6\textwidth]{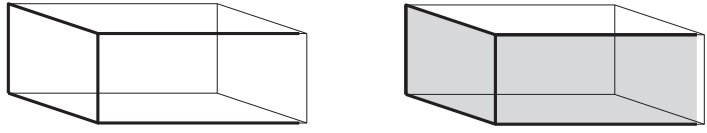}
	\caption{Left, $\alpha_m$ is the number of vertices lying on the thicker edges of the grid. Right, $\alpha_M$ is the number of vertices lying on  the gray faces of the grid.}
	\label{fig:alphas}
\end{figure}

\begin{prop}\label{prop:nokftrs-3grid} 
	If $k\ge \alpha_M(n_1,n_2,n_3)$, then	
	$P_{n_1}\Box P_{n_2}\Box P_{n_3}$ has no \kftrs{}.
\end{prop}
\begin{proof}
	Assume without loss of generality that $\alpha_M =(n_1+n_2-2)n_3$.
	Suppose to the contrary that $S\subseteq V=V(P_{n_1}\Box P_{n_2}\Box P_{n_3})$ is a {\kftrs}. Consider the partition $\{ V_1,V_2 \}$ of $V$, where $V_1=\{ u\in V: \textrm{pr}_3(u)\in R_{--}(1,1)\cup R_{++}(1,1) \}$ and $V_2=\{ u\in V: \textrm{pr}_3(u)\in R_{-+}(1,1)\cup R_{+-}(1,1) \}$. 
	Then, $|V_2|=(n_1+n_2-2)n_3=\alpha_M(n_1,n_2,n_3) \le k$.  Hence, by removing from $S$ the at most $k$ vertices belonging to $V_2$ we have a set of vertices of $P_{n_1}\Box P_{n_2}\Box P_{n_3}$ that is not resolving by Lemma~\ref{lem:regiones3D}, which is a contradiction.
\end{proof}

\begin{prop}\label{prop:bound3D} 
	If $P_{n_1}\Box P_{n_2}\Box P_{n_3}$ has a \kftrs{}, then	
			
	\vspace{-5mm}
	\begin{equation*}
	\dim_{k+1}(P_{n_1}\Box P_{n_2}\Box P_{n_3})\ge  2k+2.
	\end{equation*}

\end{prop}
\begin{proof}
	We prove that there is no {\kftrs} of cardinality at most $2k+1$.
	Let $a_1=\lfloor n_1/2 \rfloor$, $a_2=\lfloor n_2/2 \rfloor$ and suppose that $S\subseteq V(P_{n_1}\Box P_{n_2}\Box P_{n_3})$ is a {\kftrs} of cardinality at most $2k+1$. Consider the partition $\{ V_1,V_2 \}$ of the set  $V=V(P_{n_1}\Box P_{n_2}\Box P_{n_3})$ such that $V_1=\{ (i,j,h)\in V :  (i,j)\in R_{--}(a_1,a_2)\cup R_{++}(a_1,a_2) \}$ and $V_2=\{ (i,j,h) \in V: (i,j)\in R_{-+}(a_1,a_2)\cup R_{+-}(a_1,a_2)\}$. By the Pigeonhole Principle, at least one of the sets $V_1$ or $V_2$ has at most $k$ vertices of $S$. Hence, by removing these vertices from $S$ we have a set of vertices of $P_{n_1}\Box P_{n_2}\Box P_{n_3}$ that is not resolving by Lemma~\ref{lem:regiones3D}, which leads to a contradiction.
\end{proof}

\begin{prop}\label{prop:bound3Deven}  
If $P_{n_1}\Box P_{n_2}\Box P_{n_3}$ has a \kftrs{} and 	
$k$ is even, then 

\vspace{-5mm}
\begin{equation*}
\dim_{k+1}(P_{n_1}\Box P_{n_2}\Box P_{n_3})\ge  2k+3.
\end{equation*}

\end{prop}
\begin{proof}
	Let $V=V(P_{n_1}\Box P_{n_2}\Box P_{n_3})$ and let $a_1=\lfloor n_1/2 \rfloor$, $a_2=\lfloor n_2/2 \rfloor$, and $a_3=\lfloor n_3/2 \rfloor$. Suppose that $S\subseteq V$ is a {\kftrs} of cardinality at most $2k+2$. 
Consider the sets of $V$:

\vspace{-5mm}
{\small
\begin{align*}
R_{---}(a_1,a_2,a_3)&=\{ (x_1,x_2,x_3): 0\le x_1< a_1, 0\le x_2< a_2 , 0\le x_3 <a_3\};\\
R_{--+}(a_1,a_2,a_3)&=\{ (x_1,x_2,x_3): 0\le x_1< a_1, 0\le x_2< a_2 , a_3\le x_3 <n_3\};\\
R_{+--}(a_1,a_2,a_3)&=\{ (x_1,x_2,x_3): a_1\le x_1<n_1, 0\le x_2< a_2 , 0\le x_3 <a_3\};\\
R_{+-+}(a_1,a_2,a_3)&=\{ (x_1,x_2,x_3): a_1\le x_1<n_1, 0\le x_2< a_2 ,a_3\le x_3 <n_3\};\\
R_{-+-}(a_1,a_2,a_3)&=\{ (x_1,x_2,x_3): 0\le x_1< a_1, a_2\le x_2<n_2 , 0\le x_3 <a_3\};\\
R_{-++}(a_1,a_2,a_3)&=\{ (x_1,x_2,x_3): 0\le x_1< a_1, a_2\le x_2<n_2,a_3\le x_3 <n_3 \};\\
R_{++-}(a_1,a_2,a_3)&=\{ (x_1,x_2,x_3): a_1\le x_1<n_1, a_2\le x_2<n_2 , 0\le x_3 <a_3 \};\\
R_{+++}(a_1,a_2,a_3)&=\{ (x_1,x_2,x_3): a_1\le x_1<n_1, a_2\le x_2<n_2 , a_3\le x_3<n_3 \}.
\end{align*} }

Notice that, by definition, these sets form a partition of $V$. Moreover, if we consider the vertices of the grid as points of the \threedim space, these sets are included in the eight regions defined by the semispaces $x_1<a_1$, $x_1\ge a_1$; $x_2<a_2$, $x_2\ge a_2$; and $x_3< a_3$, $x_3\ge a_3$.  
Consider the following partition of $S$ (see Figure~\ref{fig:cubos}):

\vspace{-5mm}
\begin{align*}
&A_1=S\cap R_{---}(a_1,a_2,a_3), 
&A_2=S\cap R_{--+}(a_1,a_2,a_3),\\
&B_1=S\cap R_{+--}(a_1,a_2,a_3),
&B_2=S\cap R_{+-+}(a_1,a_2,a_3),\\
&C_1=S\cap R_{-+-}(a_1,a_2,a_3),
&C_2=S\cap R_{-++}(a_1,a_2,a_3),\\
&D_1=S\cap R_{++-}(a_1,a_2,a_3),
&D_2=S\cap R_{+++}(a_1,a_2,a_3).\\
\end{align*}

\begin{figure}[t!]
	\centering
	\includegraphics[width=0.5\textwidth]{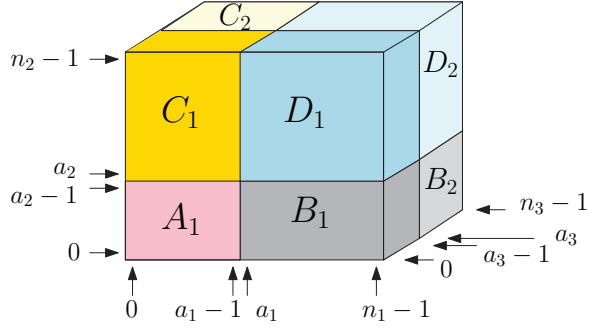}
	\caption{The sets $A_1,A_2,B_1,B_2,C_1,C_2,D_1,D_2$ are the vertices of $S$ included in the regions obtained when partitioning the vertices of $V$ into the eight octants determined by the chosen values of $a_1$, $a_2$ and $a_3$.}
	\label{fig:cubos}
\end{figure}

 If $|A_1|+|A_2|+|D_1|+|D_2|\le k$ or $|B_1|+|B_2|+|C_1|+|C_2|\le k$, then $S$ is not a {\kftrs} because, by Lemma~\ref{lem:regiones3D}, the removal of the at most $k$ vertices of $S$ in $A_1\cup A_2\cup D_1\cup D_2$ or in $B_1\cup B_2\cup C_1\cup C_2$  produces a non-resolving set, which leads to a contradiction.
Hence, 

\vspace{-5mm}
\begin{equation*}
|A_1|+|A_2|+|D_1|+|D_2|\ge  k+1\hbox{ and }|B_1|+|B_2|+|C_1|+|C_2|\ge k+1.
\end{equation*}
But, since 

\vspace{-5mm}
\begin{equation*}
|A_1|+|A_2|+|D_1|+|D_2|+|B_1|+|B_2|+|C_1|+|C_2|=|S|\le 2k+2,
\end{equation*}
we derive $|S|=2k+2$, and 

\vspace{-5mm}
\begin{equation*}
|A_1|+|A_2|+|D_1|+|D_2| =|B_1|+|B_2|+|C_1|+|C_2| = k+1.
\end{equation*}
By symmetry, we derive 

\vspace{-5mm}
\begin{align*}
|A_1|+|B_1|+|C_2|+|D_2| =|A_2|+|B_2|+|C_1|+|D_1| = k+1,\\
|A_1|+|C_1|+|B_2|+|D_2| =|A_2|+|C_2|+|B_1|+|D_1| = k+1.
\end{align*}
Hence,

\vspace{-5mm}
\begin{align}
&|A_1|+|A_2|+|D_1|+|D_2| = k+1\\
&|A_1|+|B_1|+|C_2|+|D_2| =k+1\\
&|A_1|+|C_1|+|B_2|+|D_2|=k+1\\
&|A_2|+|B_2|+|C_1|+|D_1| = k+1
\end{align}
By subtracting Equation (2) from Equation (1), Equation (3) from Equation (1) and Equation (4) from Equation (3), we get

\vspace{-5mm}
\begin{equation*}
|D_1|+|A_2|= |B_1|+|C_2|=|C_1|+|B_2|=|A_1|+|D_2|=q,
\end{equation*}
for some $q\in \mathbb{Z}$ and, consequently,

\vspace{-5mm}
\begin{equation*}
|S|= |A_1|+|A_2|+|B_1|+|B_2|+|C_1|+|C_2|+|D_1|+|D_2|= 4q.
\end{equation*}
Hence, $|S|=2k+2=4q$, which implies that $k+1$ is even, a contradiction.
\end{proof}

Note that the preceding lemma provides an alternative way of proving that the metric dimension of a 3D grid is at least 3 (see Theorem~\ref{thm:basisGrid3D}).

Next, we describe a \kftrs{} of minimum cardinality for some \threedim grids. Bearing this in mind, we define the following $n_1+n_2+n_3-4$ disjoint sets of four vertices:

\vspace{-4mm}
{\small
\begin{align*}
S_{1,i}&=\{ (i,0,0),(i,n_2-1,0),(i,0,n_3-1),(i,n_2-1,n_3-1)\},\hbox{ for }0\le i\le n_1-1,\\
S_{2,j}&=\{ (0,j,0),(n_1-1,j,0), (0,j,n_3-1), (n_1-1,j,n_3-1)\}, \hbox{ for }1\le j\le n_2-2,\\
S_{3,h}&=\{ (0,0,h),(n_1-1,0,h), (0,n_2-1,h),(n_1-1,n_2-1,h)\}, \hbox{ for }1\le h\le n_3-2.
\end{align*} }
\begin{prop}\label{prop:kftrsOdd} 	 
If $k$ is an odd integer and $k<  \alpha_m (n_1,n_2,n_3)$, then	

\vspace{-3mm}
\begin{equation*}
\dim_{k+1}(P_{n_1}\Box P_{n_2}\Box P_{n_3}) =  2k+2.
\end{equation*}
\end{prop}

\begin{proof}
By Proposition~\ref{prop:bound3D}, it is enough to construct a \kftrs{} of cardinality $2k+2$.
Let $\alpha =\tfrac{k+1}2$. Note that by hypothesis  
$\alpha=\tfrac{k+1}2\le  n_1+n_2+n_3-4 $, so we can consider the following set $S$ formed by $\alpha $ of the previously defined sets of vertices (see an example in Figure~\ref{fig:kftrs3D}, left):	

\vspace{-3mm}
\begin{align*}
S&=\bigcup_{i=0}^{\alpha-1} S_{1,i}, \hbox{ if }\alpha \le n_1,\\ 
S&=\bigg(\bigcup_{i=0}^{n_1-1} S_{1,i}\bigg)\cup \bigg(\bigcup_{j=1}^{\alpha -n_1} S_{2,j}\bigg),\hbox{ if }n_1< \alpha \le  n_1+n_2-2,\\
S&=\bigg(\bigcup_{i=0}^{n_1-1} S_{1,i}\bigg)\cup \bigg(\bigcup_{j=1}^{n_2-2} S_{2,j}\bigg)\cup \bigg(\bigcup_{h=1}^{\alpha -n_1-n_2+2} S_{3,h}\bigg),\hbox{ if }n_1+n_2-2<\alpha .\end{align*} 
	\begin{figure}[t!]
		\centering
		\includegraphics[width=\textwidth,height=0.13\textheight]{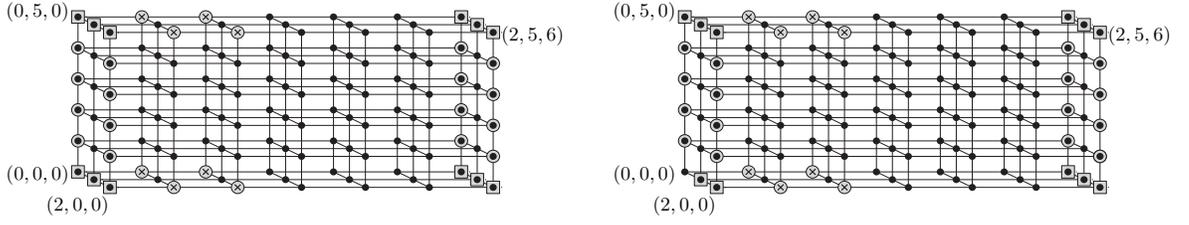}
		\caption{A 18-resolving set (left) and a 17-resolving set (right)  of the grid $P_{3}\Box P_{6}\Box P_{7}$ consists of squared, circled and crossed vertices.
		Squared vertices belong to some set $S_{1,i}$, circled vertices to some $S_{2,j}$, and crossed vertices to some $S_{3,h}$.}
		\label{fig:kftrs3D}
	\end{figure}
Obviously, $|S|=4\, \alpha =2k+2$. We claim that $S$ is a {\kftrs} of $P_{n_1}\Box P_{n_2}\Box P_{n_3}$. Indeed, any subset $S'$ obtained after the removal of $k$ vertices from $S$, contains at least three vertices of either one of the sets $S_{1,i_0}$ for some $0\le i_0\le n_1-1$, 
or $S_{2,j_0}$ for some $1\le j_0\le n_2-2$, or $S_{3,h_0}$, for some $1\le h_0\le n_3-2$, since otherwise we have to remove at least $2\alpha =k+1$ vertices.
By Lemma~\ref{lem:grid3resset4}, there exists a vertex such that together with these three vertices  form a resolving set for $P_{n_1}\Box P_{n_2}\Box P_{n_3}$. Hence, $S$ is a {\kftrs}.
\end{proof}

\begin{prop}\label{prop:kftrsEven} 
	If $k$ is an even integer and $k< \alpha_m(n_1,n_2,n_3)$, then
	
		\vspace{-3mm}
	\begin{equation*}
		\dim_{k+1}(P_{n_1}\Box P_{n_2}\Box P_{n_3}) =  2k+3.
	\end{equation*}
\end{prop}
\begin{proof} By Proposition~\ref{prop:bound3Deven}, it is enough to construct a \kftrs{} of cardinality $2k+3$. Let $\alpha=\tfrac k2 +1$. Observe that, by hypothesis, $\alpha \le n_1+n_2+n_3-4$.
	 Thus, we can consider the following set $S=S^*-\{ (0,0,0) \}$, where $S^*$ is formed by the union of $\alpha $ of the previously defined sets of vertices (see an example in Figure~\ref{fig:kftrs3D}, right):
	 
		\vspace{-5mm}	 
	 \begin{align*}
	 S^*&=\bigcup_{i=0}^{\alpha-1} S_{1,i}, \hbox{ if }\alpha \le n_1,\\ 
	 S^*&=\bigg(\bigcup_{i=0}^{n_1-1} S_{1,i}\bigg)\cup \bigg(\bigcup_{j=1}^{\alpha -n_1} S_{2,j}\bigg),\hbox{ if }n_1< \alpha \le  n_1+n_2-2,\\
	 S^*&=\bigg(\bigcup_{i=0}^{n_1-1} S_{1,i}\bigg)\cup \bigg(\bigcup_{j=1}^{n_2-2} S_{2,j}\bigg)\cup \bigg(\bigcup_{h=1}^{\alpha -n_1-n_2+2} S_{3,h}\bigg),\hbox{ if }n_1+n_2-2<\alpha .\end{align*} 

	Obviously, $|S|=4\, \alpha -1=2k+3$. We claim that $S$ is a {\kftrs} of $P_{n_1}\Box P_{n_2}\Box P_{n_3}$. 
	Indeed, any subset $S'$ obtained after the removal of $k$ vertices from $S$, contains at least three vertices of either one of the sets $S_{1,i_0}$ for some $0\le i_0\le n_1-1$, 
	or $S_{2,j_0}$ for some $1\le j_0\le n_2-2$, or $S_{3,h_0}$, for some $1\le h_0\le n_3-2$, since otherwise we have to remove at least $2(\alpha -1)+1 =k+1$ vertices.
	By Lemma~\ref{lem:grid3resset4}, there exists a vertex such that together with these three vertices  form a resolving set for $P_{n_1}\Box P_{n_2}\Box P_{n_3}$. Hence, $S$ is a {\kftrs}.
\end{proof}

Hence, the exact value of the \kftmd{} has been determined whenever $k< \alpha_m(n_1,n_2,n_3)$. In the remaining cases this parameter is defined, 
Proposition~\ref{prop:kftrsgap} provides us an upper bound on the \kftmd{}.

\begin{cor}\label{prop:kftrsOdd} 	 
	If $k<  \alpha_M(n_1,n_2,n_3) $, then $F(n_1,n_2,n_3)$ is a \kftrs{} of $P_{n_1}\Box P_{n_2}\Box P_{n_3}$ and, hence, 
	
\vspace{-5mm}
\begin{equation*}
\dim_{k+1}(P_{n_1}\Box P_{n_2}\Box P_{n_3}) \le  n_1n_2n_3- (n_1-2)(n_2-2)(n_3-2).
\end{equation*}	
\end{cor}

We can summarize the previous results as follows. 

\begin{theorem} Let $k,n_1,n_2,n_3\ge 2$.
	
	\begin{enumerate}
		\item If  $2\le k< \alpha_m(n_1,n_2,n_3)$, then
		
		\vspace{-4mm}
			\begin{equation*}
				\dim_{k+1}(P_{n_1}\Box P_{n_2}\Box P_{n_3})=\begin{cases} 2k+2, \hbox{ if $k$ is odd;}\\
					2k+3, \hbox{ if $k$ is even.}
				\end{cases}
			\end{equation*}

	\item If  $\alpha_m(n_1,n_2,n_3)\le k< \alpha_M(n_1,n_2,n_3)$, then
	
	\vspace{-4mm}
	\begin{equation*}
	\dim_{k+1}(P_{n_1}\Box P_{n_2}\Box P_{n_3})	\le n_1n_2n_3-(n_1-2)(n_2-2)(n_3-2)
\end{equation*}	

	\item  If  $ \alpha_M(n_1,n_2,n_3)\leq k$, then $P_{n_1}\Box P_{n_2}\Box P_{n_3}$ has no \kftrs{}.
	\end{enumerate}

\end{theorem}
Observe that $\alpha_m=\alpha_M$, whenever $\min\{n_1,n_2,n_3\}=2$. Hence, the value of the \kftmd{}  is completely determined for these cases.

We finish by posing a conjecture about the exact value of the \kftmd{} of 3D grids whenever $\alpha_m\le k< \alpha_M$, based on the ideas used to construct \kftrss{} for $k<\alpha_m$.

\begin{conj} 
If  $\alpha_m(n_1,n_2,n_3)\le k< \alpha_M(n_1,n_2,n_3)$, then 

\vspace{-5mm}
{
\begin{align*}
&\dim_{k+1}(P_{n_1}\Box P_{n_2}\Box P_{n_3})=\\
&=\min\{  4k-2\alpha_m(n_1,n_2,n_3)+4,n_1n_2n_3-(n_1-2)(n_2-2)(n_3-2) \}.
\end{align*}	
}
\end{conj}




\section{Acknowledgments}

M. Mora is supported by projects 
H2020-MSCA-RISE-2016-734922 CONNECT,
PID2019-104129GB-I00/MCIN/AEI/10.13039/501100011033
of the Spanish Ministry of Science and Innovation
and Gen.Cat. DGR2017SGR1336; 
M. J. Souto-Salorio is supported by project PID2020-113230RB-C21 of the Spanish Ministry of Science and Innovation.




\end{document}